\documentclass[reqno]{amsart}
\usepackage{amssymb,amsmath}
\usepackage{amsthm}
\usepackage{color,graphicx}
\usepackage{xcolor}

\newcommand{\X}{\mathcal{X}_T^s}
\newcommand{\Y}{\mathcal{X}_T^s}
\newcommand{\lank}{\langle \xi \rangle}

\newcommand{\Ph}{\Phi_\mu (\xi)}
\newcommand{\ha}{\hat{\phi}}
\newcommand{\x}{(1+\xi^2)}

\newcommand{\les}{\lesssim}

\newcommand{\R}{\mathbb R}

\newcommand{\p}{\partial}

\newtheorem{theorem}{Theorem}[section]
\newtheorem{proposition}[theorem]{Proposition}
\newtheorem{remark}[theorem]{Remark}
\newtheorem{lemma}[theorem]{Lemma}

\begin{document}
\vglue-1cm \hskip1cm
\title[The Kuramoto-Sivashinsky equation]{The IVP for the Kuramoto-Sivashinsky equation in low regularity Sobolev spaces}



\author[A. Cunha]{Alysson Cunha}
\address{Universidade Federal de Goi\'as, Instituto de Matem\'atica e Estat\'istica.
Universidade Federal de Goi\'as - UFG - Campus II, Goi\^ania, Bra\-zil.}
\email{alysson@ufg.br}

\author[E. Alarcon]{Eduardo Alarcon}
\address{Universidade Federal de Goi\'as, Instituto de Matem\'atica e Estat\'istica.
Universidade Federal de Goi\'as - UFG - Campus II, Goi\^ania, Bra\-zil.}
\email{arbietoalarcon@ufg.br}



\begin{abstract}

In this work, we study the initial-value problem associated with the Kuramoto-Sivashinsky equation. We show that the associated initial value problem is locally and globally well-posed in Sobolev spaces $H^s(\mathbb{R})$, where $s>1/2$. We also show that our result is sharp, in the sense that the flow-map data-solution is not $C^2$ at origin, for $s<1/2$. 
Furthermore, we study the behavior of the solutions when $\mu\downarrow 0$. 
\end{abstract}

\maketitle

\section{Introduction}\label{introduction}

This paper is concerned with the initial-value problem (IVP) for the Kuramoto-Sivashinsky equation (KS)

\begin{equation}\label{ks}
\begin{cases}
u_{t}-\partial_{x}^2 u-\mu (1-\p_x^2)^{-1/2}u-\frac 12 (\p_x u)^2=0, \;\;x\in\R, \;t\geq 0, \\
u(x,0)=\phi(x),
\end{cases}
\end{equation}

where $\mu>0$ is a constant and $u$ is a real-valued function. 

First we present a derivation of the equation \eqref{ks}. Indeed, an initial value problem equivalent to (\ref{ks})

\begin{equation}\label{ksd}
 \begin{cases}
\partial_t {H }-\mathcal{D}_C\,\partial_{x}^2 H -\frac 12 (\p_x H)^2 - \dfrac{\delta\,G}{8 \pi}\displaystyle \int_{-\infty}^{+\infty}\, \int_{-\infty}^{+\infty} e^{i\,k (x - x^\prime)} \dfrac{H(x^\prime,\,t)}{\sqrt{\dfrac{1}{4} + k^2}} \,dx^\prime\,dk =0, \\
u(x,0)=\phi(x),
\end{cases}
\end{equation}
where $\mathcal{D}_C$ is dimensionless catalyst diffusivity, $\delta$ is relative density and $G$ is dimensionless acceleration of gravity, was derived by G. I. Sivashinsky, et al (\cite{Siv}), to describe vertical propagation of chemical waves fronts in the presence instability due to density gradients (possibly thermally induced). Assuming an interaction region thin enough to be described as a surface ($ z = H(x,\,y,\,t)$), where $H$  is the vertical position of the front, they use thermo-hydrodynamic equations in the regions of reacted fluid ( $ z < H(x,\,y,\,t)$) and unreacted fluid ( $z > H(x,\,y,\,t)$) together with conservation of energy, matter and momentum to derive jump conditions on discontinuities at the interface. The equations governing these autocatalytic systems involving propagating reaction-diffusion fronts have been derived in (\cite{Edw}), where they consider the reaction front to be very thin chemically, other assumption in use involves how the densities of the fluids change with temperature. Since the density changes due to thermal expansion of the fluids are small, write the density of the fluids to first order as $\rho(T) = \rho_1[ 1 - \alpha(T - T_1)]$, where $\rho(T)$ is the density at temperature $T$, $\rho_1$ is density at the reference temperature $T_1$ and $\alpha$ is the classical thermal expansion coefficient at constant pressure. The relative difference between the densities of these two fluids at the front is one of the key parameters in this study, and is defined by $\delta = \dfrac{\rho^a - \rho^b}{\rho^b}$, where $\rho^a$ and $ \rho^b$ are the densities of the fluid above the front (unreacted fluid) and of that below the front (reacted fluid), respectively. This is due to the fact that $\rho$ is dependent on the thermal diffusivity of the
fluids, in (\cite{Siv}) the diffusivity is assumed to be infinite.

As in (\cite{Edw}), they obtain the following system of equations
\begin{equation}\label{sistema}
\begin{array}{l}
  \dfrac{\partial \mathbf{V}}{\partial t} + (\mathbf{V}\cdot \nabla)\,\mathbf{V}= - \dfrac{1}{\rho}\,\nabla \mathrm{P}_r + \nu \nabla^2 \mathbf{V},\\
  \nabla \cdot \mathbf{V} =0,  \\
  c= \hat{\mathbf{n}} \cdot \hat{\mathrm{z}}\,\dfrac{ \partial H}{\partial} - \hat{\mathbf{n}}\cdot \mathbf{V}|_{z=H},
\end{array}
\end{equation}
where $\mathrm{V}$ is fluid velocity, $z=H(x,\,y,\,t)$ is the vertical position of the front as a function of $x$, $y$ and $t$, $\hat{\mathbf{n}}$ unit vector pointing normal to the front into the unreacted fluid, $c$ the normal front velocity with respect to the unreacted fluid and $\nu$ is kinematic viscosity.

Together with jump conditions across the interface between the reacted and unreacted fluids given by:
\begin{enumerate}
 \item $ [\hat{\mathbf{n}}\cdot \mathbf{V} ]^{+}_{-} = 0$,
\item $ [\hat{\mathbf{n}}\times \mathbf{V}]^{+}_{-}  = 0$,
 \item $ [\mathbf{P}_r]^{+}_{-} - [n_i\,n_j\,T^\mathrm{V}_{ij}]^{+}_{-} = - \delta\,\rho\,g\,H$,
 \item $ [\epsilon_{ijk}\,n_j\,n_l\,T^\mathbf{V}_{kl}]^{+}_{-} = 0$,
\end{enumerate}
 and viscous stress tensor $ T_{ij}^\mathbf{V} = - \nu\,\rho\,\left( \dfrac{\partial V_i}{\partial x_j} + \dfrac{\partial V_j}{\partial x_i}\right)$, $P_r$ is the reduced pressure given by $P_r = P + \rho\,g\,z$ end $\epsilon_{ijk}$ is the totally antisymmetric tensor.

 By making asymptotic expansions in the delta parameter of the variables  $\mathbf{V}$, $\mathbf{P}_r$ and $H$ , we obtain equation $\ref{ksd}$ (see \cite{Siv}).

 As the usual, we are assuming the well-posedness in the Kato's sense, that is, includes, existence, uniqueness, persistence property and smoothness of the map data-solution, see \cite{APweigh}, \cite{APlow}, \cite{Iorio1}, \cite{Kato} and \cite{KP}. In \cite{AlarOtter} and \cite{cunha} the authors, using the Banach fixed point theorem obtained the local and global well-posedness for the IVP \eqref{ks}. More precisely they proved the following theorems.
\medskip

\noindent {\bf Theorem A} (Local well-posedness) {\em
Let $s\geq 1$. Then for any $\phi\in H^{s}(\R),$ there exists a positive
$T=T_\mu(\|\phi\|_{H^{s}})$  and a unique solution $u\in C([0,T];H^{s}(\R))$ of
the $\mathrm{IVP}$ \eqref{ks}. Furthermore, the flow-map $\phi\mapsto u(t)$
is continuous in the $H^{s}$-norm.}
\medskip

\noindent {\bf Theorem B} (Global well-posedness) {\em
The problem \eqref{ks} is globally well posed in $H^s (\R)$, for $s\geq 1$.}
\medskip

In this paper, we are mainly interested in improving the last two theorems. For this we look at the dissipative effect of the IVP \eqref{ks} and we will use the same methods of Dix \cite{Dix} (see also \cite{Beki}).
In general terms, the Dix method consists of an application of the fixed point theorem in a suitable time-weighted function space. Recently, many authors have used this technique, see, for example, Carvajal and Panthee \cite{pan1,pan2}, Esfahani \cite{amin}, Fonseca, Pastr\'an and Rodr\'iguez-Blanco \cite{pastran} and Pilot \cite{didier}. See also \cite{pan}. We observe that, in these works the highest order dissipative term of the equations, often is greater than or equal to three (unless in \cite{pan2}, which is greater than $5/2$). In our work, the degree of the highest dissipation term is two, this show, in particular, that this technique is also useful when we have a low order of dissipation.


Our results are sharp in the sense that the flow-map data-solution, for the IVP \eqref{ks}, is not $C^2$ at origin, for $s<1/2$. As it is well known, a consequence of this fact is that the Cauchy problem \eqref{ks}, for $s<1/2$, cannot solve by a contraction argument on the integral equation (see \cite{Bour}, \cite{luc1}, \cite{luc2}, \cite{Tzvet} and references therein). 


Now we state the main results of this paper.

\begin{theorem}(Local well-posedness).\label{localmeio}
Let $\mu>0$ and $s>1/2,$ then for all $\phi \in H^s(\R),$ there exists $T=T(\|\phi\|_{H^s})$, a space
$$\X \hookrightarrow C([0,T];H^s(\R))$$
and a unique solution $u$ of \eqref{ks} in $\X$. In addition, the flow map data-solution
$$S: H^s(\R)\to \X \cap C([0,T];H^s), \phi \mapsto u$$
is smooth and
$$u\in C((0,T]; H^{\infty}(\R)).$$
Moreover, if $s'>s$ then the solution with initial data $\phi \in H^{s'}(\R)$ is defined in the same interval $[0,T]$, with $T=T(\|\phi\|_{H^s})$.
\end{theorem}

\begin{theorem}(Global well-posedness).\label{global}
Let $\mu>0$ and $s>1/2$, then the initial value problem \eqref{ks} is globally well-posed in $H^s(\R)$.
\end{theorem}

\begin{theorem}(Ill-posedness).\label{Illks}
Let $s<1/2$, if there exists some $T>0$, such that the problem \eqref{ks} is locally well-posed in $H^s(\R)$, then the flow-map data solution
$$S:H^s (\R)\to C([0,T];H^s(\R)), \phi \mapsto u,$$
is not $C^2$ at zero.
\end{theorem}

An open problem about the IVP \eqref{ks} is to investigate the existence of a global attractor (see \cite{attAlarcon}, \cite{attAlarcon2} and \cite{attAlarcon1}). In view of the ideas in \cite{Ab}, \cite{attAlarcon}, \cite{attAlarcon1}, \cite{luc3} and \cite{Teman} we believe that it's possible to show the existence of the global attractor in $H^s(\R)$, where $s>1/2$. For a general theory about the global attractor, see \cite{Ab} and \cite{Teman}.

Another interesting question would be to explore the well-posedness for the IVP \eqref{ks} with bore-like data. That is, we deal with the problem \eqref{ks}, with $g$ instead of $\phi$, where $g$ satisfies

\begin{itemize}
\item [i)] $g(x)\to C_{\pm}$ with $x\to \pm \infty$;
\item [ii)] $g' \in H^s$, for some $s>1/2$;
\item [iii)] $(g-C_{\pm})\in L^2 ([0,\infty))$ and $(g-C_{\pm})\in L^2 (-\infty,0])$;
\item [iv)] $\p_x^{-1} g$ and $\p_x^{-1}g'\in H^{3/2}(\R),$
\end{itemize}
and $\widehat{\p_x^{-1}g}(\xi)=\displaystyle\frac{\hat{g}(\xi)}{i\xi}$.

In particular, these results on bore-like data, would improve those obtained in \cite{bore}. More information on bore-like data, can be found in \cite{Iorio1}.

This paper is organized as follows. In the next section, we derive some preliminary estimates. The well-posedness for the IVP \eqref{ks}, for $s>1/2$, is established in section 3. In section 4 we deal with the limit when $\mu \downarrow 0$. Finally, in section 5 we state the results about Ill-posedness for the IVP \eqref{ks}.

\subsection{Notation}
In this article, we use the following notation. We say $a \lesssim b$ if there exists a constant $c>0$ such that $a\leq c b$.
 By $a\thicksim b$ we mean that $a\lesssim b$ and $b\lesssim a$. We write $a \les_{l} b$  when the constant depends on only parameter $l$. The Fourier transform of $f$, is defined by
 $$\hat{f}(\xi)=\int_{\R} e^{-i\xi x}f(x)dx.$$

 If $s\in \R$, $H^s:=H^s(\R)$ represents the nonhomogeneous Sobolev space defined as
 $$H^s(\R)=\{f\in \mathcal{S}'(\R): \|f\|_{H^s}<\infty\},$$
 where
 $$\|f\|_{H^s}=\|\langle \xi \rangle^s \hat{f}\|_{L^2_\xi},$$
 and $\langle \xi \rangle=(1+\xi^2)^{1/2}$.

In addition, we define the Bessel potential $J^s$ by
$$(J^s f)^{\wedge}(\xi)=\langle \xi \rangle^s\hat{f}(\xi), \ \mbox{for all} \ f\in \mathcal{S}'(\R),$$
hence $\|J^s f\|_{L^2_x}=\|f\|_{H^s}$.

In the rest of the paper, we will denote the $L^2$-norm in the $x$ variable by $\|\cdot\|_{L^2_x}:=\|\cdot\|$.


\section{Preliminary estimates}

By defining $$\Phi(\xi)=-\xi^2 +\frac{\mu}{(1+\xi^2)^{1/2}},$$ the semigroup associated with the linear part of \eqref{ks} is defined via Fourier transform by
\begin{equation}\label{semi}
E_\mu (t)\psi=\Big(e^{t\Phi(\xi)}\hat{\phi}(\xi)\Big)^{\vee},
\end{equation}
and the integral equation associated to \eqref{ks}
\begin{equation}\label{integral}
F_\mu(u)(t):=u(t)=E_\mu(t)\phi+\frac 12 \int_0^t E_\mu (t-\tau)(\p_x u)^2 (\tau)d\tau.
\end{equation}

The following result is useful in establishing of estimates for the semigroup $E_\mu$.
\begin{lemma}\label{lemma1}
 Let $\mu>0$, $\lambda\geq 0$, $\nu \in \R$ and $t\in [0,T]$. Then

\begin{equation}\label{lemalambda}
\|\xi^{2\lambda}e^{t(-\xi^2 +\frac{\mu}{(1+\xi^2)^{1/2}})}\|_{L^{\infty}_\xi}\leq e^{\mu T}\Big(\frac{\lambda}{e}\Big)^{\lambda} t^{-\lambda}
\end{equation}
and
\begin{equation}\label{lemanu}
\||\xi|^\nu e^{-t\xi^2}\|_{L^2_{\xi}}=c_\nu t^{-\frac{\nu}{2}-\frac{1}{4}}.
\end{equation}
\end{lemma}
\begin{proof} First we will establish \eqref{lemalambda}.
For this, note that for all $\xi \in\R$ and $t\in [0,T]$
$$\xi^{2\lambda} e^{t\Phi(\xi)}\leq e^{\mu t}\xi^{2\lambda}e^{-t\xi^2}.$$

Therefore, looking at the maximum value of function $(\cdot)^{2\lambda}e^{-t(\cdot)^2}$, we obtain the result.

About identity \eqref{lemanu}, by using the change of variables $\xi=t^{-1/2}w$

$$\||\xi|^\nu e^{-t\xi^2}\|^2_{L^2_\xi}=t^{-\nu-1/2} \int w^{2\nu}e^{-2w^2}dw=c_\nu^2 t^{-\nu-1/2}.$$
This finish the proof.

\end{proof}

In the following, we deal with the well-posedness for the IVP \ref{ks}, where $s>1/2$. First, we need a technical lemma, which will be useful in our linear estimates. This is a new version of Lemma 2.1 of \cite{pan2}. 
\begin{lemma}\label{M1}
There exists $M>0$ such that if $|\xi|\geq M$, then
\begin{equation}\label{phi}
\Phi(\xi)=-\xi^2+\mu \lank^{-1}<-1
\end{equation}
and
\begin{equation}\label{phi1}
|\Phi(\xi)|\geq \frac{\xi^2}{2}.
\end{equation}
\end{lemma}
\begin{proof}
The inequalities \eqref{phi} and \eqref{phi1} follows from
$$\lim_{|\xi|\to \infty}\frac{\Phi(\xi)}{\xi^2}=-1$$
and
$$\lim_{|\xi|\to \infty}\frac{\mu \lank^{-1}}{\xi^2}=0.$$
\end{proof}
The next lemma is a simple result about calculus.
\begin{lemma}\label{calculus}
Let $f(t)=t^\alpha e^{t \beta},$ $\alpha>0$ and $\beta<0$. Then, for all $t\geq 0$
$$f(t)\leq \Big(\frac{\alpha}{|\beta|}\Big)^{\alpha}e^{-\alpha}.$$
\end{lemma}
\begin{proof}
See Lemma 2.3 of \cite{pan2}.
\end{proof}
Next, we present the function spaces appropriated for to show the existence of a solution.

Let $0<s<1$ and $0<T\leq 1$, then we define
\begin{equation}
\X=\{u\in C([0,T];H^s): \|u\|_{\X}<\infty\},
\end{equation}
where
\begin{equation}\label{X}
\|u\|_{\X}:=\sup_{[0,T]}\Big(\|u(t)\|_{H^s}+t^{\frac{1-s}{2}}\|\p_x u(t)\|_{L^2}\Big).
\end{equation}
 In the following, we present the linear estimates in the spaces $\X$. The proof follows the same ideas contained in Lemma 2.6 of \cite{pan2}.
\begin{lemma}\label{linearks}
Let $\mu>0,$ $0<T\leq 1,$ $s<1$, $t\in [0,T]$ and $\phi \in H^s(\R)$. Then
\begin{equation}
\|E_{\mu}(t)\phi\|_{\X}\leq C\|\phi\|_{H^s},
\end{equation}
where $C$ depends on $s$, $\mu$, $T$ and $M$, with $M$ as in Lemma \ref{M1}.
\end{lemma}
\begin{proof}
By the definition of semigroup $E_\mu$, the first term in \eqref{X} can be estimated as follows
\begin{equation}
\begin{split}\label{first1}
\|E_\mu(t)\phi\|_{H^s}&\leq \|e^{t\Phi(\xi)}\lank^s \ha\|\\
&\leq \|e^{t\Phi(\xi)}\|_{L^\infty_\xi}\|\phi\|_{H^s}\\
&\lesssim e^{\mu T}\|\phi\|_{H^s}.
\end{split}
\end{equation}
For estimate the second term of the $\X$-norm, putting $\alpha=\frac{1-s}{2}$ and using the Plancherel identity, we have
\begin{equation}
\begin{split}\label{second}
t^{\frac{1-s}{2}}\|\p_x E_\mu(t)\phi\|_{L^2}&= t^{\frac{1-s}{2}} \|\xi e^{t\Phi(\xi)} \ha\|\\
&= t^{\frac{1-s}{2}} \|\xi \lank^{-s}e^{t\Phi(\xi)}\lank^s \ha\|_{L^2_\xi}\\
&\leq t^\frac{1-s}{2} \|\xi \lank^{-s}e^{t\Phi(\xi)}\|_{L^\infty}\|\phi\|_{H^s}.
\end{split}
\end{equation}
We can write
\begin{equation}
\begin{split}\label{dec}
\|\xi \lank^{-s}e^{t\Phi(\xi)}\|_{L^\infty}\leq&\|\lank^{1-s}e^{t\Phi(\xi)} \chi_{\{|\xi|\leq M\}}\|_{L^\infty_\xi}+ \|\lank^{1-s}e^{t\Phi(\xi)} \chi_{\{|\xi|\geq M\}}\|_{L^\infty_\xi}\\
&:=C_M+I,
\end{split}
\end{equation}
then
\begin{equation}\label{M}
t^\frac{1-s}{2} C_M  \leq C_M.
\end{equation}
By using Lemmas \ref{calculus} and \ref{M}, with $\alpha=\frac{1-s}{2}$ and $\beta=\Phi(\xi)$, follows that
\begin{equation}
\begin{split}\label{final1}
t^\alpha I&=t^\alpha \|\lank^{1-s}e^{t\Phi(\xi)} \chi_{\{|\xi|\geq M\}}\|_{L^\infty_\xi}\\
&\leq \|\lank^{1-s}\Big(\frac{\alpha}{|\Phi(\xi)|}\Big)^{\alpha}e^{-\alpha}\chi_{\{|\xi|\geq M\}}\|_{L^\infty_\xi}\\
&\leq (2\alpha)^{\alpha} \|\lank^{1-s}\xi^{-2\alpha}\chi_{\{|\xi|\geq M\}}\|_{L^\infty_\xi}\\
&\leq (2\alpha)^{\alpha} \|(\xi^{-2\alpha}+|\xi|^{1-s}\xi^{-2\alpha})\chi_{\{|\xi|\geq M\}}\|_{L^\infty_\xi}\\
&\leq (2\alpha)^{\alpha}(M^{-2\alpha}+1),
\end{split}
\end{equation}
where above, we used
$$\lank^{1-s}\lesssim 1+|\xi|^{1-s}.$$
Therefore, by \eqref{first1}--\eqref{final1}, we conclude the proof.
\end{proof}
Next, we deduce some bilinear estimates useful to proof the Theorem \ref{localmeio}.
\begin{proposition}\label{pro4}
Let $\mu>0,$ $0<T\leq 1$ and $1/2<s<1$. Then
$$\Big\| \int_{0}^{t} E_\mu (t-\tau)(\p_x u \p_x v)(\tau)d\tau\Big\|_{\X} \lesssim e^{\mu T} T^{\frac{s}{2}+\frac14}\|u\|_{\X}\|v\|_{\X},$$
for all $u,v \in \X$.
\end{proposition}
\begin{proof}
 Let $0<t\leq T$, therefore by the definition of norms $\X$ we obtain
\begin{equation}
\tau^{\frac{1-s}{2}}\| \p_x u\|_{L^2}\leq \|u\|_{\X} \ \mbox{and} \ \tau^{\frac{1-s}{2}}\|\p_x v\|_{L^2}\leq \|v\|_{\X}.
\end{equation}
Since $s>0$, we see that $(1+\xi^2)^{s/2}\lesssim 1+|\xi|^s$, then
\begin{equation}
\begin{split}\label{first}
\Big \| \int_{0}^{t} E_\mu (t-\tau)(\p_x u \p_x v)(\tau)d\tau\Big\|_{H^s}\leq& \Big \| \int_{0}^{t} E_\mu (t-\tau)(\p_x u \p_x v)(\tau)d\tau\Big\|_{L^2}+\\
&+ \Big\| \int_{0}^{t} E_\mu (t-\tau)(\p_x u \p_x v)(\tau)d\tau\Big\|_{\dot{H}^s}\\
&:=A+B.
\end{split}
\end{equation}

Therefore, by Young's inequality for convolution and identity \eqref{lemanu}
\begin{equation}
\begin{split}\label{choosealpha}
B&\leq \int_{0}^{t} \||\xi|^s e^{-(t-\tau)\Ph}(\p_x u \p_x v (\tau))^{\wedge}(\xi)\|_{L^2_\xi}d\tau\\
&\leq e^{\mu T} \int_{0}^{t} \||\xi|^{s} e^{-(t-\tau)\xi^2}\|_{L^2_\xi}\|\widehat{\p_x u}\ast \widehat{\p_x v} \|_{L^\infty_\xi}d\tau\\
&\lesssim_s e^{\mu T} t^{\frac{s}{2}-\frac14}\int_0^1 (1-\sigma)^{-\frac{s}{2}-\frac{1}{4}}\sigma^{s-1} d\sigma \|u\|_{\X}\|v\|_{\X},
\end{split}
\end{equation}
and
\begin{equation}
\begin{split}\label{A}
A&\leq 
 e^{\mu T} \int_{0}^{t} \|e^{-(t-\tau)\xi^2}\|_{L^2_\xi}\|\widehat{\p_x u}\ast \widehat{\p_x v} \|_{L^\infty_\xi}d\tau\\
&\lesssim e^{\mu T} t^{s-\frac12}\int_0^1 (1-\sigma)^{-\frac{1}{4}}\sigma^{s-1} d\sigma \|u\|_{\X}\|v\|_{\X}.
\end{split}
\end{equation}
In the above, again we used the change of variables $\sigma=\frac{\tau}{t}$.

With respect to second norm in \eqref{X}
\begin{equation}
\begin{split}\label{B}
t^{\frac{1-s}{2}}\Big \| \int_{0}^{t} E_\mu (t-\tau)\p_x(uv)(\tau)d\tau \Big\| & \leq e^{\mu T} t^{\frac{1-s}{2}}\int_0^t \|e^{-(t-\tau)\xi^2}(\p_x u\p_x v)^{\wedge}(\xi,\tau)\|d\tau\\
&\lesssim e^{\mu T} t^{\frac{1-s}{2}}\int_{0}^{t} \|e^{-(t-\tau)\xi^2}\|_{L^2_\xi}\|\widehat{\p_x u}\ast \widehat{\p_x v}\|d\tau\\
&\lesssim e^{\mu T} t^{\frac{1-s}{2}}\int_{0}^{t} (t-\tau)^{-\frac14}\|\p_x u\|\|\p_x v\|d\tau\\
&\lesssim e^{\mu T} t^{\frac{1-s}{2}}\int_{0}^{t} (t-\tau)^{-\frac14}\tau^{s-1}d\tau \|u\|_{\X}\|v\|_{\X}\\
&\lesssim e^{\mu T}T^{\frac{s}{2}+\frac{1}{4}}\int_0^1 (1-\sigma)^{-\frac14}\sigma^{s-1} d\sigma \|u\|_{\X}\|v\|_{\X}.
\end{split}
\end{equation}

Therefore, by \eqref{first}--\eqref{B} we obtain the proof.

\end{proof}
The next result will be useful to obtain regularity of the solutions, in Theorem \ref{localmeio}.
\begin{proposition}\label{prop3meio}
Let $\mu>0,$ $0<T\leq 1,$ $1/2<s<1$ and $\lambda \geq 0$. If $s+\lambda <3/2$ and $\lambda<s-1/2$,
then the application
\begin{equation}
W_\mu: t\in [0,T]\longmapsto \int_{0}^t E_{\mu}(t-\tau)(\p_x u)^2(\tau)d\tau \in H^{s+\lambda}(\R),
\end{equation}
is continuous.
\end{proposition}
\begin{proof}
Fixed $0\leq t_0 <t \leq T$, then

\begin{equation}
\begin{split}\label{decompmeio}
W_\mu(t)-W_\mu(t_0)=&\int_{0}^{t_0}(E_\mu (t-\tau)-E_\mu (t_0-\tau))(\p_x u)^2 (\tau)d\tau+\\
&+\int_{t_0}^{t}E_\mu (t -\tau)(\p_x u)^2 (\tau)d\tau\\
&:=\psi_1(t,t_0)+\psi_2(t,t_0).
\end{split}
\end{equation}

Case a): $s+\lambda\geq 0$. In view of $(1+\xi^2)^{\frac{s+\lambda}{2}}\lesssim 1+|\xi|^{s+\lambda}$, we obtain

\begin{equation}
\begin{split}
\|\psi_2(t,t_0)\|_{H^{s+\lambda}}&\leq \int_{t_0}^{t}\|E_\mu (t-\tau)(\p_x u)^2 (\tau)\|d\tau+\int_{t_0}^{t}\|E_\mu (t-\tau)(\p_x u)^2 (\tau)\|_{\dot{H}^{s+\lambda}}d\tau\\
&:=A(t,t_0)+B(t,t_0).
\end{split}
\end{equation}
Then

\begin{equation}
\begin{split}
A(t,t_0)&\leq e^{\mu T} \int_{t_0}^{t} \|e^{-(t-\tau)\xi^2}\|_{L^2_\xi} \|\widehat{\p_x u}\ast \widehat{\p_x u}\|d\tau\\
          &\lesssim e^{\mu T} \int_{t_0}^{t}(t-\tau)^{-1/4}\tau^{s-1} d\tau\|u\|^2_{\X}\\
            &\lesssim (t-t_0)^{s-1/4}e^{\mu T} \int_{0}^{1}(1-\sigma)^{-1/4}\sigma^{s-1} d\sigma\|u\|^2_{\X}\to 0,
\end{split}
\end{equation}
with $t\to t_0$.

As for the other integral
\begin{equation}
\begin{split}\label{samesteps}
B(t,t_0)&\lesssim e^{\mu T} \int_{t_0}^{t} \||\xi|^{s+\lambda}e^{-(t-\tau)\xi^2}\|_{L^2_\xi} \|\widehat{\p_x u}\ast \widehat{\p_x u}\|d\tau\\
          &\lesssim e^{\mu T} \int_{t_0}^{t}(t-\tau)^{-\frac{s+\lambda}{2}-\frac{1}{4}}\tau^{s-1} d\tau\|u\|^2_{\X}\\
            &\lesssim e^{\mu T} (t-t_0)^{\frac{s-\lambda}{2}+\frac14}\int_{0}^{1}(t-\tau)^{-\frac{s+\lambda}{2}-\frac14}\tau^{s-1} d\tau\|u\|^2_{\X}\\
            &\to 0,
\end{split}
\end{equation}
with $t\to t_0$.
Where, in the above arguments we used the change of variables $\tau=t_0+\sigma(t-t_0)$  and the inequality $\tau^s<\sigma^s(t-t_0)^s$.


With respect to the first integral in \eqref{decompmeio}
\begin{equation}
\begin{split}\label{intcmeio}
\|\psi_1(t,t_0)\|_{H^{s+\lambda}}\leq & \|\psi_1(t,t_0)\| + \int_{t_0}^{t}\||\xi|^{s+\lambda}(E_\mu (t -\tau)-E_\mu (t_0-\tau))(\p_x u)^2 (\tau)\| d\tau \\
&:= E(t,t_0)+F(t,t_0).
\end{split}
\end{equation}
Then from way analogous to the above case $E(t,t_0)\to 0$, with $t\to t_0$.

With respect to second integral

$$F(t,t_0)\lesssim_{\mu,T} \Big ( \int_{0}^{t_{0}} \|h(t,t_0,\tau,\xi)\|\tau^{s-1} d\tau \Big)\|u\|_{\X}^2,$$
where
$$h(t,t_0,\tau,\xi):=|\xi|^{s+\lambda}\Big(e^{-(t-\tau)\xi^2}-e^{-(t_0-\tau)\xi^2}\Big)
.$$
Note that the function $g$ above converges to zero, with $t$ goes to $t_0$, for all $\xi \in \R$. We also have by the inequality $t-\tau>t_0-\tau$ that $|h(t,t_0,\tau,\xi)|\leq 2 e^{-(t_0-\tau)\xi^2}|\xi|^{s+\lambda} \in L^2_\xi.$
Then by the Lebesgue dominated convergence theorem
$$\|h\|_{L^2_\xi}\to 0, \ \mbox{with}\ t\to t_0.$$
Moreover
\begin{equation}
\begin{split}
\|h(t,t_0,\tau,\xi)\|_{L^2_\xi}\tau^{s-1} &\leq 2 \||\xi|^{s+\lambda}e^{-(t_0-\tau)\xi^2}\|\tau^{s-1}\\
                    &\lesssim (t_0-\tau)^{-\frac{s+\lambda}{2}-\frac{1}{4}}\tau^{s-1}\in L^1_\tau (0,t_0).
\end{split}
\end{equation}
A new application of the Lebesgue dominated convergence theorem gives us
$$F(t,t_0)\to 0, \ \mbox{with} \ t\to t_0.$$

This concludes the proof of case a).

Case b): $s+\lambda<0$. Since $\dot{H}^{s+\lambda}\hookrightarrow H^{s+\lambda}$, we obtain

\begin{equation}
\begin{split}
\|\psi_2(t,t_0)\|_{H^{s+\lambda}}
          &\lesssim e^{\mu T} \int_{t_0}^{t} \||\xi|^{s+\lambda} e^{-(t-\tau)\xi^2}\|_{L^2_\xi} \|\widehat{\p_x u}\ast \widehat{\p_x v}\|_{L^\infty_\xi}d\tau\\
          &\lesssim e^{\mu T} \int_{t_0}^{t}\||\xi|^{s+\lambda} e^{-(t-\tau)\xi^2}\|_{L^2_\xi}\frac{d\tau}{\tau^{1-s}}\|u\|^2_{\X}\\
          &\lesssim_{s,\lambda} e^{\mu T} \int_{t_0}^{t}(t-\tau)^{-\frac{s+\lambda}{2}-\frac14}\tau^{s-1} d\tau\|u\|^2_{\X}\\
          &\lesssim_{s,\lambda} e^{\mu T} (t-t_0)^{\frac{s-\lambda}{2}-\frac14}\int_{0}^{1}(1-\sigma)^{-\frac{s+\lambda}{2}-\frac14}\sigma^{s-1} d\sigma\|u\|^2_{\X}\\
          &\to 0,
\end{split}
\end{equation}
with $t\to t_0. $

Finally, we see that $\psi_1(t,t_0)$ converges to zero in $H^{s+\lambda}$-norm, by way analogous to the term $F(t,t_0)$, in \eqref{intcmeio}.
Therefore, the proof of proposition is finalized.
\end{proof}

\begin{remark}\label{s'}
Let $s'>s>1/2$, then modifying the space $\mathcal{X}_{T}^{s'}$ by
$$\tilde{\mathcal{X}}_{T}^{s'}=\{u\in \mathcal{X}_{T}^{s'}; \|u\|_{\tilde{\mathcal{X}}_{T}^{s'}}<\infty\}$$
with
$$\|u\|_{\tilde{\mathcal{X}}_{T}^{s'}}=\|u\|_{\mathcal{X}_{T}^{s'}}+\sup_{t\in [0,T]}\Big(t^{s/2}\|J^{s'-s}u(t)\|_{L^2}\Big)$$
and using the fact that
$$(1+\xi^2)^{s'/2}\lesssim (1+\xi^2)^{s/2}(1+\xi_{1}^2)^{(s'-s)/2}+(1+\xi^2)^{s/2}\big(1+(\xi-\xi_1)^2\big)^{(s'-s)/2},$$
we obtain, from way similar to the Proposition \ref{pro4}
$$\Big\|\int_0^t E_\mu(t-\tau)(\p_x u \p_x v)(\tau)d\tau\Big\|_{\tilde{\mathcal{X}}_{T}^{s'}}\lesssim_{s} e^{\mu T}T^{\delta(s)} (\|u\|_{\tilde{\mathcal{X}}_{T}^{s'}}\|v\|_{\mathcal{X}_{T}^{s}}+\|u\|_{\mathcal{X}_{T}^{s}}\|v\|_{\tilde{\mathcal{X}}_{T}^{s'}}).$$
\end{remark}

\section{Proof of Theorems \ref{localmeio}--\ref{global}}
\begin{proof}[Proof of Theorem \ref{localmeio}]
Let $\mu>0$ and $1/2<s<1$. Our strategy is to show that the operator $F_\mu$ given by \eqref{integral} is a contraction in some closed ball in $\Y$. In fact, by \eqref{integral}
\begin{equation}\label{closed}
\begin{split}
\|F_\mu (t)\|_{\Y}& \les \|E_\mu(t)\phi\|_{\Y}+\Big \| \int_{0}^{t} E_\mu (t-\tau)(\p_x u)^2 (\tau)d\tau\Big\|_{\Y} \\
                   & \les C_{s,\mu}(\|\phi\|_{H^s}+T^{\delta(s)} \|u\|^2_{\Y}),
\end{split}
\end{equation}
and

\begin{equation}\label{cont}
\begin{split}
\|F_\mu (u)-F_\mu(v)\|_{\Y}&\leq \|\int_{0}^{t} E_\mu (t-\tau)((\p_x u)^2- (\p_x v)^2)(\tau) d\tau\|_{\Y}\\
&=\|\int_{0}^{t} E_\mu (t-\tau)(\p_x(u-v)\p_x(u+v))(\tau) d\tau\|_{\Y}\\
&\leq C_{s,\mu}T^{\delta(s)} \|u-v\|_{\Y}\|u+v\|_{\Y},
\end{split}
\end{equation}
for all $u,v\in \Y$ and $0<T\leq 1$.

Therefore, given $R=2C_{s,\mu}\|\phi\|_{H^s}$ we define
\begin{equation}\label{ball}
\Y (R)=\{u\in \Y; \|u\|_{\Y}\leq R\}.
\end{equation}

Then taking
\begin{equation}\label{T}
0<T<\min \Big \{ (4C_{s,\mu}^2 \|\phi\|_{H^s})^{-1/\delta(s)},1  \Big \},
\end{equation}
the estimates \eqref{closed} and \eqref{cont} implies that $F_\mu$ is a contraction on the $\Y (R)$. Then by the Banach fixed point theorem, there exists a unique solution $u$ of the integral equation \eqref{integral} in $\Y (R)$. By the Proposition \ref{prop3meio} follows that $u\in C([0,T];H^s(\R))$. The uniqueness in whole space $\Y$ and the smoothness of the flow-map solution follows by know arguments, see for example \cite{AlarOtter}, \cite{pastran} and \cite{HBIS}.

Let $s'>s$, then a similar contraction argument using the norm $\tilde{\mathcal{X}}_{T}^{s'}$, defined in Remark \ref{s'}, shows that the solution with initial data $\phi \in H^{s'}$ is defined on $[0,T]$ with $T=T(\|\phi\|_{H^s})$.

With respect to regularity, we note that $t\in (0,T]\longmapsto E_\mu (t)\phi \in H^{\infty}(\R)$ is continuous with respect to the topology of $H^{\infty}$, see \cite{AlarOtter} and \cite{cunha}. From the Proposition \ref{prop3meio} there exists $\lambda>0$ such that  $V_\mu \in C([0,T];H^{s+\lambda} (\R)),$ thus
$$u\in C((0,T]; H^{s+\lambda}(\R).$$

Therefore, by a well known bootstrapping argument, using the uniqueness result and the fact that $T$ only depends on the $H^s$-norm of the initial data, we obtain
                      $$u\in C((0,T];H^\infty(\R)).$$

\end{proof}

\begin{proof}[Proof of Theorem \ref{global}]

Let $u$ be the local solution given by Theorem \ref{localmeio}. In view of $u\in C((0,T]; H^{\infty})$, we only need an \textit{a priori} estimate in $H^1$. For this, putting $w=\partial_{x}u,$ we get the following
\begin{eqnarray}\label{3.9}
\left\{\begin{array} {lcc}
\partial_{t}w-\partial_{x}^{2}w-w\partial_{x}w-\mu(1-\partial_{x}^{2})^{-1/2}w=0\\
w(x,0)=\phi'(x).
\end{array} \right.
\end{eqnarray}
Multiplying \eqref{3.9} by $w$ and integrating over the real line we obtain
$$\dfrac{d}{dt}\|w\|^{2}=-2\|\partial_{x}w\|^{2}+2\mu\|w\|_{H^{-1/2}}^{2}
\leq 2\mu\|w\|^{2}.$$
Then by the Gronwall's Lemma
\begin{equation}\label{3.12}
\|\partial_{x}u\|^{2}\leq \|\phi'\|^{2}\,\exp\left[2\mu\int_{0}^{t}dt'\right]\leq e^{2\mu T}\,\|\phi'\|^{2}.
\end{equation}

In the following, multiplying  \eqref{ks} by $u$ and integrating over $\R$ we get
\begin{equation}
\begin{split}\label{3.20}
\dfrac{d}{dt}\|u\|^{2}&= \ 2\langle u,\partial_{x}^{2}u\rangle+\langle u,(\partial_{x}u)^{2}\rangle+2\mu\langle u,(1-\partial_{x}^{2})^{-1/2}u\rangle\\
&\leq \ -2\|\partial_{x}u\|^{2}+\|u\|_{L^{\infty}}\|\partial_{x}u\|^{2}+2\mu\|u\|^{2}\\
&\lesssim  \ \|u\|^{\frac{1}{2}}\|\partial_{x}u\|^{\frac{5}{2}}+2\mu\|u\|^{2}\\
&\lesssim_{\mu,T}  (\|\phi'\|^{\frac{10}{3}}+\|u\|^{2}),
\end{split}
\end{equation}
 where above, we use respectively, the Gagliardo-Nirenberg's inequality, and the Young's inequality. Therefore, an application of Gronwall's Lemma in \eqref{3.20} give us the desired result. This finish the proof.

\end{proof}

\section{Convergence of solutions when $\mu \downarrow 0$}
In this section we study the behavior of the solutions of IVP \eqref{ks}, when $\mu$ goes to zero. In the following, we define by $u_\mu$, the solution of the IVP \eqref{ks} constructed in Theorem \ref{ks}, on parameter $\mu>0$ and defined in the interval $[0,T]$. Recall that by the proof of Theorem \ref{ks}, $T$ is independent of $\mu$. Here we are using arguments similar to \cite{HBIS} (see also \cite{didier}).

\begin{theorem}\label{themu}
Let $1/2<s<1$, $\phi \in H^s(\R)$ and $\mu>0$. If $u_\mu$ is the solution defined as above, for $u_\mu(0)=\phi$, then
\begin{equation}
\lim_{\mu\downarrow 0}\sup_{t\in [0,T]}\|u_\mu-u\|_{H^s}=0,
\end{equation}
where $u$ is the solution of \eqref{ks}, on parameter $\mu=0$, with $u(0)=u_0(0)=\phi$.
\end{theorem}
\begin{proof}
Putting $E:=E_0$, after straightforward computations, follows that $w:=u_\mu-u$ satisfies the integral equation
\begin{equation}
\begin{split}\label{intmu}
w=&(E_\mu (t)-E(t))\phi+\frac12\int_0^t E_\mu(t-\tau)[\p_xu_\mu\p_x w+\p_x u\p_x w]d\tau\\
  &+\frac12 \int_0^t (E_\mu(t-\tau)-E(t-\tau))(\p_x u)^2 d\tau\\
  &:=\alpha_1+\alpha_2+\alpha_3.
\end{split}
\end{equation}

Let $u_\mu$ solutions constructed in Theorem \ref{ks}, on interval $[0,T_{0}]$, such that $$\|u_\mu\|_{\mathcal{X}^s_{T_{0}}}\leq R, \ \mbox{where}\ 0<T_{0}\leq \min\{1,T\}.$$
In view of inequality
\begin{equation}\label{conv}
\sup_{t\in [0,T_{0}]}\|u_\mu-u\|_{H^s}\leq \|w\|_{\mathcal{X}^s_{T_{0}}},
\end{equation}
 is enough to examine the convergence on spaces $\mathcal{X}^s_{T_{0}}$.
Then, by using Proposition \ref{pro4}
\begin{equation}
\begin{split}
\|\alpha_2\|_{\mathcal{X}^s_{T_{0}}}&\lesssim e^{\mu T}T_{0}^{\delta(s)}(\|u_\mu\|_{\mathcal{X}^s_{T_{0}}}\|w\|_{\mathcal{X}^s_{T_{0}}}+\|u\|_{\mathcal{X}^s_{T_{0}}}\|w\|_{\mathcal{X}^s_{T_{0}}})\\
&\lesssim 2 e^{\mu T}T_{0}^{\delta(s)}R \|w\|_{\mathcal{X}^s_{T_{0}}}.
\end{split}
\end{equation}
Taking $T_{0}$ such that $2e^{\mu T}T_{0}^{\delta(s)}R<1/2$, follows by \eqref{intmu}
\begin{equation}
\begin{split}\label{varphi2}
\|w\|_{\mathcal{X}^s_{T_{0}}}&\leq \|\alpha_1\|_{\mathcal{X}^s_{T_{0}}}+\|\alpha_2\|_{\mathcal{X}^s_{T_{0}}}+\|\alpha_3\|_{\mathcal{X}^s_{T_{0}}}\\
                               &\leq \|\alpha_1\|_{\mathcal{X}^s_{T_{0}}}+\frac12 \|w\|_{\mathcal{X}^s_{T_{0}}}+\|\alpha_3\|_{\mathcal{X}^s_{T_{0}}}.
\end{split}
\end{equation}
Then
\begin{equation}\label{varphi}
\|w\|_{\mathcal{X}^s_{T_{0}}}\leq 2(\|\alpha_1\|_{\mathcal{X}^s_{T_{0}}}+\|\alpha_3\|_{\mathcal{X}^s_{T_{0}}}).
\end{equation}
By the last inequality its enough study the limit on $\varphi_1$ and $\varphi_3$. For this, we observe that by the definition of the  ${\mathcal{X}^s_{T_{0}}}$-norms and the Lebesgue dominated convergence theorem, follows that
\begin{equation}\label{varphi1}
\|\alpha_1\|_{\mathcal{X}^s_{T_{0}}}\to 0, \ \mbox{with}\ \mu\downarrow 0.
\end{equation}
About $\alpha_3$, again by the Lebesgue dominated convergence theorem and using the same ideas as in the Propositions \ref{pro4} and \ref{prop3meio}
\begin{equation}\label{varphi3}
\|\alpha_3\|_{\mathcal{X}^s_{T_{0}}}\to 0, \ \mbox{with}\ \mu\downarrow 0.
\end{equation}
Therefore, by \eqref{conv}--\eqref{varphi3} we conclude that
\begin{equation}
\sup_{t\in [0,T_{0}]}\|u_\mu-u\|_{H^s}\to 0, \ \mbox{with}\ \mu \downarrow 0.
\end{equation}
To conclude, we can use an interactive process to extend the solution for all interval $[0,T]$. This finish the proof.
\end{proof}
\begin{remark}\label{0}
By a modification of the space $\mathcal{X}^s_{T}$ we can show the existence of solutions to the IVP \eqref{ks}, when $s>0$. In this case, the uniqueness of IVP \eqref{ks} fail, once that's in \cite{Dix}, the author obtained the non-uniqueness in the initial value problem for the Burgers' equation, where $s<-1/2$.
\end{remark}

\section{Ill-Posedness}
In this section we use analogous arguments contained in \cite{didier} (see also \cite{pastran}, \cite{pan2}, \cite{amin} and \cite{EP2}).
\begin{theorem}\label{ill5}
Let $s<1/2$ and $T>0.$ Then there no exists a space $\Y$ continuously embedded in $C([0,T];H^s(\R))$ such that
\begin{equation}\label{ill1}
\|E_\mu (t)\phi\|_{\Y}\lesssim \|\phi\|_{H^s}, \ \forall \phi\in H^s(\R)
\end{equation}
and
\begin{equation}\label{ill2}
\Big\| \int_{0}^{t} E_\mu (t-\tau)b(u,v)(\tau)d\tau\Big\|_{\Y}\lesssim\|u\|_{\Y}\|v\|_{\Y}, \forall u,v\in \Y,
\end{equation}
where
\begin{equation}
b(u,v)=\frac{1}{2}\p_x u \p_x v.
\end{equation}
\end{theorem}
\begin{proof}

The proof follows by a contradiction argument. Suppose that there exists a space $\Y$ as in theorem \ref{ill5}. Let $u(t)=E_\mu (t)\phi$ and $v(t)=E_\mu (t) \psi$ where $\phi, \psi \in H^s(\R)$ and $0<t<T$ is fixed. Using \eqref{ill1} and \eqref{ill2}

\begin{equation}\label{ill3}
\begin{split}
\Big\| \int_{0}^{t} E_\mu (t-\tau)b(u(\tau),v(\tau))d\tau\Big\|_{\Y}
                                                                     & \lesssim \|\phi\|_{H^s}\|\psi\|_{H^s}
\end{split}
\end{equation}
Now, we will construct functions $\phi$ and $\psi$ such that \eqref{ill3} fails.
Let $\phi$ and $\psi$ defined by
\begin{equation}
\phi=r^{-1/2}N^{-(s-1/2)}(\chi_{I_1})^{\vee} \ \mbox{and} \ \psi= r^{-1/2}N^{-(s-1/2)}(\chi_{I_2})^{\vee},
\end{equation}
where

$$I_1=[-N,-N+r], I_2=[N+r,N+2r], \ N>>1 \ \mbox{and} \ r\thicksim 1.$$
We observe that
\begin{equation}\label{ill4}
\|\phi\|_{\Y}\lesssim 1 \ \mbox{and} \ \|\psi\|_{\Y}\lesssim 1.
\end{equation}

Recalling that $$\Phi(\xi)=-\xi^2 +\frac{\mu}{(1+\xi^2)^{1/2}},$$ by taking the Fourier transform  and using Fubini's Theorem
\begin{equation}
\begin{split}
f(\xi,t)&:=\Big(\int_{0}^{t} E_\mu(t-\tau)b(u(\tau),v(\tau))d\tau\Big)^{\wedge}(\xi)\\
&=\int_{0}^{t} e^{(t-\tau)\Phi(\xi)}b(u(\tau),v(\tau))^{\wedge}(\xi)d\tau\\
&=\frac{1}{2}\int_{0}^{t} e^{(t-\tau)\Phi(\xi)} \underbrace{i\xi e^{\tau \Phi(\xi)}\hat{\phi}\ast i\xi e^{\tau \Phi(\xi)}\hat{\psi}}_{g(\xi,\tau)}d\tau,
\end{split}
\end{equation}
where, by the change of variables $z=\xi-\eta$ we obtain
\begin{equation}
\begin{split}
g(\xi,\tau)&:=-\int (\xi-\eta)e^{\tau \Phi(\xi-\eta)}\hat{\phi}(\xi-\eta)\eta e^{\tau\Phi(\eta)}\hat{\psi}(\eta)d\eta\\
         &=-\int ze^{\tau \Phi(z)}\hat{\phi}(z)\eta e^{\tau\Phi(\eta)}\hat{\psi}(\eta)d\eta.
\end{split}
\end{equation}

Then the integral above can be written
\begin{equation}
\begin{split}
f(\xi,t)&=-\frac{1}{2}\int_{0}^{t} e^{(t-\tau)\Phi(\xi)}\int z \eta e^{\tau \Phi(z)}\hat{\phi}(z)e^{\tau \Phi(\eta)}\hat{\psi}(\eta)d\eta d\tau\\
&=-\frac{e^{t\Phi(\xi)}}{2} \int z \eta \Big( \int_{0}^{t} e^{-\tau\Phi(\xi)} e^{\tau\Phi(z)}\hat{\phi}(z) e^{\tau\Phi(\eta)}\hat{\psi}(\eta) d\tau\Big) d\eta\\
&=-\frac{e^{t\Phi(\xi)}}{2r N^{2(s-1/2)}} \int_{K_\xi} z\eta \chi_{I_1}(z)\chi_{I_2}(\eta)(\int_{0}^{t} e^{\tau(-\Phi(\xi)+\Phi(z)+\Phi(\eta))}d\tau)d\eta\\
&=-\frac{e^{t\Phi(\xi)}}{2r N^{2(s-1/2)}} \int_{K_\xi} z\eta \left(\frac{e^{t(-\Phi(\xi)+\Phi(z)+\Phi(\eta))}-1}{-\Phi(\xi)+\Phi(z)+\Phi(\eta)} \right)d\eta,
\end{split}
\end{equation}
where $K_\xi=\{\eta\in \R: z\in I_1 \ \mbox{and} \ \eta\in I_2\}.$

We observe that if $\eta \in I_1$ and $z\in I_2$, then $|\eta|\thicksim |z|\thicksim N$, $r\leq \xi \leq 3r$ and $|\eta z|\thicksim N^2$.
Therefore, we obtain $|f(\xi,t)|^2 \gtrsim N^{-4(s-1/2)},$ so that  $$\x^{s}|f(\xi,t)|^2 \gtrsim \x^s N^{-4(s-1/2)}.$$ Thus
\begin{equation}
\begin{split}\label{ill5}
\Big\| \int_{0}^{t} E_\mu (t-\tau)b(u(\tau),v(\tau))d\tau\Big\|_{H^s}^2
&\gtrsim  N^{-4(s-1/2)}.
\end{split}
\end{equation}
Then from \eqref{ill2}, \eqref{ill4} and \eqref{ill5} follows that
$$N^{-2(s-1/2)}\lesssim 1, \ \ \forall N>>1,$$
which is a contradiction, taking account our hypothesis on $s$.


 \end{proof}

\begin{proof}[Proof of Theorem \ref{Illks}]
If the flow-map data solution would be $C^2$ at origin, by a computation of the Fr\'echet derivative we would obtain

\begin{equation*}
\Big\| \int_{0}^{t} E_\mu (t-\tau)b(u(\tau),v(\tau))d\tau\Big\|_{\Y}\lesssim  \|\phi\|_{H^s}\|\psi\|_{H^s}, \ \forall \phi,\psi \in H^s(\R).
\end{equation*}
But as we have seen in \eqref{ill3} the above inequality fails. This finish the proof.
\end{proof}


\bibliographystyle{mrl}

\end{document}